\documentclass{article}

\usepackage[T1]{fontenc}
\usepackage[utf8]{inputenc}

\usepackage{subcaption}

\usepackage{adjustbox}

\usepackage{dsfont} 
\usepackage{amssymb}

\usepackage{amsmath}
\usepackage{amsthm}
\newtheorem{theorem}{Theorem}
\newtheorem{lemma}[theorem]{Lemma}

\newtheorem{proposition}[theorem]{Proposition}

\usepackage[nocompress]{cite}

\usepackage{tikz}
\usetikzlibrary{
    calc,
    decorations.pathreplacing,
    intersections,
    positioning,
    backgrounds,
}
\usepackage{pgfplots}
\pgfplotsset{compat=1.14}

\pgfplotsset{
    closed dot/.style = {
        only marks,
        mark = *,
        mark size = 1.5pt,
    },
    open dot/.style = {
        only marks,
        mark = *,
        fill = white,
        mark size = 1.5pt,
    },
}

\newcommand{\floor}[1]{\left\lfloor #1 \right\rfloor}
\newcommand{\ceil}[1]{\left\lceil #1 \right\rceil}
\newcommand{\norm}[1]{\left\lVert #1 \right\rVert}

\DeclareMathOperator{\aff}{aff} 
\DeclareMathOperator{\lcm}{lcm} 

\usepackage[hidelinks]{hyperref}

\makeatletter
\newcommand{\restate@theorem@name}{}
\newtheorem*{restate@theorem}{\restate@theorem@name}
\newenvironment{restatetheorem}[1]{
    \renewcommand{\restate@theorem@name}{Theorem~\ref{#1}}
    \begin{restate@theorem}
}{
    \end{restate@theorem}
}

\begin{document}

\title{Semi-reflexive polytopes}
\author{
    Tiago Royer
}
\date{}
\maketitle

\begin{abstract}
    The Ehrhart function $L_P(t)$ of a polytope $P$
    is usually defined only for integer dilation arguments $t$.
    By allowing arbitrary real numbers as arguments
    we may also detect integer points
    entering (or leaving) the polytope in fractional dilations of $P$,
    thus giving more information about the polytope.
    Nevertheless,
    there are some polytopes that only gain new integer points
    for integer values of $t$;
    that is,
    these polytopes satisfy $L_P(t) = L_P(\floor t)$.
    We call those polytopes \emph{semi-reflexive}.
    In this paper,
    we give a characterization of these polytopes
    in terms of their hyperplane description,
    and we use this characterization to show that
    a polytope is reflexive
    if and only if both it and its dual are semi-reflexive.
\end{abstract}

\section{Introduction}

Given a polytope $P \subseteq \mathbb R^d$,
the classical Ehrhart lattice point enumerator $L_P(t)$ is defined as
\begin{equation*}
    L_P(t) = \#(tP \cap \mathbb Z^d), \qquad \text{integer $t > 0$.}
\end{equation*}
Here,
$\#(A)$ is the number of elements in $A$
and $tP = \{tx \mid x \in P\}$ is the dilation of $P$ by $t$.
The above definition may be extended
to allow arbitrary real numbers as dilation parameters;
we will assume this extension in this paper.
Moreover,
we will agree that $L_P(0) = 1$.

To minimize confusion,
we will denote real dilation parameters with the letter $s$,
so that $L_P(t)$ denotes the classical Ehrhart function
and $L_P(s)$ denotes the extension considered in this paper.
So,
for example,
$L_P(t)$ is just the restriction of $L_P(s)$ to integer arguments.

Every polytope $P \subset \mathbb R^d$ may be written as
\begin{equation}
    \label{intersection-representation}
    P = \bigcap_{i = 1}^n \{
            x \in \mathbb R^d \mid \langle a_i, x \rangle \leq b_i
        \},
\end{equation}
where each $a_i$ is a vector in $\mathbb R^d$
and each $b_i$ is a real number.

We define a polytope $P$ to be \emph{semi-reflexive}
if $P$ is rational and satisfies $L_P(s) = L_P(\floor s)$
for all $s \geq 0$.
We have the following characterization.

\begin{theorem}
    \label{thm:characterization}
    Let $P$ be a rational polytope.
    Then $P$ is a semi-reflexive polytope
    if and only if $P$ may be written as in~\eqref{intersection-representation},
    with all the $a_i$ being integers
    and all $b_i$ being either $0$ or $1$.
\end{theorem}

Section~\ref{sec:examples} contains some examples of semi-reflexive polytopes.
The characterization above is proven in Section~\ref{sec:characterization}.
A similar proof yields another, similar,
characterization of semi-reflexive polytopes
in terms of their relative interiors;
this is discussed in Section~\ref{sec:interiors-of-polytopes}.
Finally,
in Section~\ref{sec:relation-with-reflexive}
we show how semi-reflexive polytopes relate with reflexive polytopes.

\section{Examples of semi-reflexive polytopes}
\label{sec:examples}

Here,
we will use the ``if'' part of the characterization
to provide some examples of semi-reflexive polytopes.

\begin{itemize}
    \item The unit cube $P = [0, 1]^d$.
        $P$ may be represented by the inequalities
        $x_i \geq 0$ and $x_i \leq 1$, for all $i$.

    \item The standard simplex,
        which is defined by the inequalities $x_i \geq 0$ for all $i$,
        and $x_1 + \dots + x_d \leq 1$.

    \item The cross-polytope.
        This polytope is defined by $|x_1| + \dots |x_d| \leq 1$.
        Each of the $2^d$ vectors $(\alpha_1, \dots \alpha_d) \in \{-1, 1\}^d$
        gives a bounding inequality of the form
        $\alpha_1 x_1 + \dots + \alpha_d x_d \leq 1$,
        all of which satisfy the characterization.

    \item Order polytopes~\cite{StanleyCounterexample}.
        Let $\prec$ be a partial order over the set $\{1, \dots, d\}$.
        The order polytope for the partial order $\prec$
        is the set of points $(x_1, \dots, x_d) \in \mathbb R^d$
        which satisfy $0 \leq x_i \leq 1$ for all $i$,
        and $x_i \leq x_j$ whenever $i \prec j$.

    \item Chain polytopes~\cite{StanleyCounterexample}.
        Let $\prec$ be a partial order over the set $\{1, \dots, d\}$.
        The chain polytope for the partial order $\prec$
        is the set of points $(x_1, \dots, x_d) \in \mathbb R^d$
        which satisfy $0 \leq x_i$ for all $i$,
        and $x_{i_1} + \dots + x_{i_k} \leq 1$
        for all chains $i_1 \prec i_2 \prec \dots \prec i_k$.

    \item Quasi-metric polytopes~\cite{Cris2017}.
        Let $G$ be a graph such that every vertex has degree $1$ or $3$.
        Identify its edge set with $\{1, \dots, d\}$.
        The quasi-metric polytope $\mathcal P_G$
        is defined to be the set of points $(x_1, \dots, x_d) \in \mathbb R^d$
        satisfying all inequalities of the form
        $x_i \leq x_j + x_k$ and $x_i + x_j + x_k \leq 1$,
        whenever $i$, $j$ and $k$ are the edges
        incident to a degree-$3$ vertex in $G$.

    \item And,
        as we will see in Section~\ref{sec:relation-with-reflexive},
        all reflexive polytopes are also semi-reflexive.
\end{itemize}

\section{
    Characterizing semi-reflexive polytopes
    in terms of their hyperplane description
}
\label{sec:characterization}
(We will use the Iverson bracket in this section.)

In this section,
mostly deal with full-dimensional polytopes.
If $P$ is full-dimensional,
then in the representation~\eqref{intersection-representation}
the number $n$ may be chosen to be the number of facets in $P$.
In this case,
each hyperplane $\{ x \mid \langle a_i, x \rangle = b_i \}$ intersects $P$ in a facet,
so that there is no redundant hyperplanes;
that is, such representation is minimal.
We'll assume such representations are always minimal for full-dimensional polytopes.

For full-dimensional polytopes,
there is a slight generalization
of the hyperplane description of Theorem~\ref{thm:characterization}.
Instead of demanding that $a_i$ is an integer for all $i$,
we will only require that $a_i$ is an integer if the corresponding $b_i$ is $1$.
Thus, if $b_i = 0$, then $a_i$ may be an arbitrary vector.
Observe that,
for rational polytopes,
$a_i$ will necessarily be a rational vector,
so when $b_i = 0$ we may multiply both sides by an appropriate integer constant
(the $\lcm$ of all denominators of the coordinates of $a_i$)
to get the same condition of Theorem~\ref{thm:characterization}.

One direction is easy.

\begin{theorem}
    \label{thm:if-part}
    Let $P$ be as in~\eqref{intersection-representation}.
    If all $b_i$ is either $0$ or $1$,
    and $a_i$ is integral whenever $b_i = 1$,
    then $L_P(s) = L_P(\floor s)$.
\end{theorem}

\begin{proof}
    \begin{equation*}
        L_P(s)
            = \sum_{x \in \mathbb Z^d} [x \in sP]
            = \sum_{x \in \mathbb Z^d} \prod_{i = 1}^n
                \big[\langle a_i, x \rangle \leq s b_i\big].
    \end{equation*}

    If $b_i = 0$, the term $[\langle a_i, x \rangle \leq s b_i]$
    reduces to $[\langle a_i, x \rangle \leq 0]$,
    which is constant for all $s$;
    thus $[\langle a_i, x \rangle \leq s b_i]
    = \big[\langle a_i, x \rangle \leq \floor s b_i\big]$.

    If $b_i = 1$,
    as $x$ and $a_i$ are integral,
    the number $\langle a_i, x \rangle$ is an integer,
    thus $[\langle a_i, x \rangle \leq s b_i]
    = \big[\langle a_i, x \rangle \leq \floor s b_i\big]$ again.

    Therefore,
    \begin{align*}
        L_P(s)  &= \sum_{x \in \mathbb Z^d} \prod_{i = 1}^n
                    \big[\langle a_i, x \rangle \leq s b_i\big] \\
                &= \sum_{x \in \mathbb Z^d} \prod_{i = 1}^n
                    \big[\langle a_i, x \rangle \leq \floor s b_i\big] \\
                &= \sum_{x \in \mathbb Z^d} [x \in \floor s P] \\
                &= L_P(\floor s).
                \qedhere
    \end{align*}
\end{proof}

For the other direction we need a lemma.

Denote the open ball with radius $\delta$ centered at $x$ by $B_\delta(x)$;
that is,
if $x \in \mathbb R^d$,
we have
\begin{equation*}
    B_\delta(x) = \{y \in \mathbb R^d \mid \norm{y - x} < \delta\}.
\end{equation*}

\begin{lemma}
    \label{thm:integral-points-in-cone}
    Let $K$ be a full-dimensional cone with apex $0$,
    and let $\delta > 0$ be any value.
    Then there are infinitely many integer points $x \in K$
    such that $B_\delta(x) \subset K$.
\end{lemma}

That is,
there are many points which are ``very inside'' $K$.

\begin{proof}
    Choose $x$ to be any rational point in the interior of the cone.
    By definition,
    there is some $\varepsilon > 0$ with $B_\varepsilon(x) \subseteq K$.
    For any $\lambda > 0$,
    we have
    \begin{equation*}
        \lambda B_\varepsilon(x) = B_{\lambda \varepsilon}( \lambda x ) \subseteq K.
    \end{equation*}

    For all sufficiently large $\lambda$,
    we have $\lambda \varepsilon > \delta$,
    so we just need to take the infinitely many integer $\lambda$
    such that $\lambda x$ is an integer vector.
\end{proof}

We'll need the fact that these integral points are distant from the boundary
only in the proof of Theorem~\ref{thm:only-if-part}.
For the next theorem,
existence of infinitely many such points would be enough.

\begin{proposition}
    \label{thm:weak-only-if-part}
    Let $P$ be a full-dimensional polytope.
    If $0 \notin P$,
    then $L_P(s)$ is not a nondecreasing function.
    In fact, $L_P(s)$ has infinitely many ``drops'';
    that is,
    there are infinitely many points $s_0$
    such that $L_P(s_0) > L_P(s_0 + \varepsilon)$
    for sufficiently small $\varepsilon > 0$.
\end{proposition}

\begin{proof}
    Writing $P$ as in~\eqref{intersection-representation},
    we conclude at least one of the $b_i$ must be negative
    (because the vector $0$ satisfies every linear restriction where $b_i \geq 0$).
    Equivalently (dividing both sides by $b_i$),
    there is a half-space of the form $\{x \mid \langle u, x \rangle \geq 1\}$
    such that some facet $F$ of $P$ is contained in
    $\{x \mid \langle u, x \rangle = 1\}$.

    Now, consider the cone $\bigcup_{\lambda > 0} \lambda F$.
    The previous lemma says
    there is an integral point $x_0$ in this cone,
    and so by its definition we have $x_0 \in s_0 F$ for some $s_0 > 0$;
    thus, $x_0 \in s_0 P$.
    We'll argue that, for all sufficiently small $\varepsilon > 0$,
    the polytopes $(s_0 + \varepsilon)P$ do not contain any integral point
    which is not present in $s_0 P$ (Figure~\ref{fig:polytope-not-containing-origin}).

    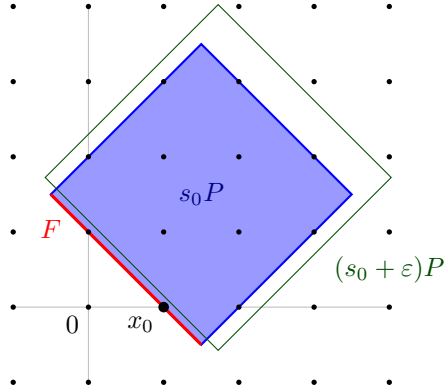
\begin{figure}[t]
        \centering
        \begin{tikzpicture}
            \def\polytope{(1.5, -0.5) -- ++(-2, 2) -- ++(2, 2) -- ++(2, -2) -- cycle}
            \filldraw[blue, thick, fill = blue!40] \polytope;
            \node [blue!50!black] at (1.5, 1.5) {$s_0 P$};

            \draw[scale = 1.15, green!30!black] \polytope;
            \node[green!30!black] at (4, 0.5) {$(s_0+\varepsilon) P$};

            \draw[very thick, red] (1.5, -.5) -- ++(-2, 2)
                node[below=0.2cm] {$F$};

            \begin{pgfonlayer}{background}
                \draw[lightgray] (-1, 0) -- (4, 0)
                            (0, -1) -- (0, 4);
            \end{pgfonlayer}{background}
            \foreach \x in {-1, ..., 4}
                \foreach \y in {-1, ..., 4}
                    \fill (\x, \y) circle [radius = 1pt];
            \node [below left] at (0, 0) {$0$};

            \fill (1, 0) circle [radius=2pt]
                node [below left] {$x_0$};
        \end{tikzpicture}

        \caption[
            If the polytope is dilated it will lose the point $x_0$.
        ]{
            The polytope $P$,
            when dilated,
            loses the point $x_0$,
            but if the dilation is small enough
            then it does not gain any new integral point.
            Therefore, $L_P(s)$ will decrease from $s_0$ to $s_0 + \varepsilon$.
        }
        \label{fig:polytope-not-containing-origin}
    \end{figure}

    For small $\varepsilon$ (say, $\varepsilon < 1$),
    all these polytopes are ``uniformly bounded'';
    that is, there is some $N$ such that $(s_0 + \varepsilon)P \subseteq [-N, N]^d$
    for all $0 \leq \varepsilon < 1$.
    Each integral $x$ in $[-N, N]^d$ which is not in $s_0 P$
    must violate a linear restriction of the form $\langle a_i, x \rangle \leq s_0 b_i$;
    that is, $\langle a_i, x \rangle > s_0 b_i$ for some $i$.
    As we're dealing with real variables,
    we also have $\langle a_i, x \rangle > (s_0 + \varepsilon) b_i$
    for all sufficiently small $\varepsilon$,
    say, for all $\varepsilon < \delta_x$ for some $\delta_x > 0$.
    Now, as there is a finite number of such relevant integral $x$,
    we can take $\delta$ to be the smallest of all such $\delta_x$;
    then if $0 < \varepsilon < \delta$
    every integral point of $(s_0 + \varepsilon)P$ also appears in $s_0 P$.

    But the special restriction $\langle u, x \rangle \geq 1$,
    considered above,
    ``dilates'' to the linear restriction
    $\langle u, x \rangle \geq s_0 + \varepsilon$ in $(s_0 + \varepsilon)P$.
    Since $x_0$ satisfy this restriction with equality for $\varepsilon = 0$,
    for any $\varepsilon > 0$ we'll have $x_0 \notin (s_0 + \varepsilon)P$.
    Therefore,
    not only the dilates $(s_0 + \varepsilon)P$ do not contain new integral points
    (for small enough $\varepsilon$),
    but actually these dilates lose the point $x_0$ if $\varepsilon > 0$.
    Thus,
    $L_P(s_0) > L_P(s_0 + \varepsilon)$ for all sufficiently small $\varepsilon > 0$.

    Finally,
    there are infinitely many integral $x_0$ in the cone
    $\bigcup_{\lambda > 0} \lambda F$;
    thus we have infinitely many different values of $\norm{x_0}$
    and hence of $s_0$,
    and the reasoning above shows $L_P(s)$ ``drops'' in every such $s_0$.
\end{proof}

A simple consequence of this proposition is that
if $s_0$ is a point where $L_P(s)$ ``drops'',
then in the interval $\big[\!\floor{s_0}, \floor{s_0} + 1\big)$
the function $L_P(s)$ will not be a constant function,
so we cannot have $L_P(s) = L_P(\floor s)$ for all $s$
if $P$ does not contain the origin.

We're now able to show that,
for full-dimensional polytopes,
the converse of Theorem~\ref{thm:if-part} holds.

\begin{theorem}
    \label{thm:only-if-part}
    Let $P$ be a full-dimensional polytope.
    If $L_P(s) = L_P(\floor s)$,
    then the polytope $P$ can be written as in~\eqref{intersection-representation}
    where each $b_i$ is either $0$ or $1$,
    and when $b_i = 1$ the vector $a_i$ must be integral.
\end{theorem}

\begin{proof}
    Since $0 \in P$, by the previous proposition,
    we know we can write $P$ as in~\eqref{intersection-representation},
    with each $b_i$ being nonnegative.
    By dividing each inequality by the corresponding $b_i$ (if $b_i \neq 0$),
    we may assume each $b_i$ is either $0$ or $1$.
    Now we must show that when $b_i = 1$,
    the vector $a_i$ will be integral.

    Let $\langle u, x \rangle \leq 1$ be one of the linear restrictions where $b_i = 1$.
    Using an approach similar to the proof of the previous proposition,
    we'll show $u$ must be an integral vector.

    Let $F$ be the facet of $P$ which is contained in $\{\langle u, x \rangle = 1\}$,
    and again define $K = \bigcup_{\lambda \geq 0} \lambda F$.
    Suppose $u$ has a non-integer coordinate;
    first, we'll find an integral point $x_0$ and a non-integer $s_0 > 0$
    such that $x_0 \in s_0 F$.

    By Lemma~\ref{thm:integral-points-in-cone}
    (using $\delta = \frac32$),
    there is some integral $y \in K$ such that $B_{\frac32}(y) \subset K$.
    If $\langle u, y \rangle$ is not an integer,
    we may let $x_0 = y$ and $s_0 = \langle u, y \rangle$;
    then $x_0$ is an integral point which is in the relative interior of $s_0 F$.
    If $\langle u, y \rangle$ is an integer,
    as $u$ is not an integral vector,
    some of its coordinates is not an integer,
    say the $j$th;
    then $\langle u, y + e_j \rangle$ will not be an integer,
    so (as $y + e_j \in B_{\frac32}(y) \subset K$)
    we may let $x_0 = y + e_j$ and $s_0 = \langle u, y + e_j \rangle$
    to obtain our desired integral point which is in a non-integral dilate of $F$.

    We have $sP \subset s_0 P$ for $s < s_0$,
    but $x_0 \notin sP$ for any $s < s_0$,
    so $L_P(s) < L_P(s_0)$ for all $s < s_0$.
    As $s_0$ is not an integer (by construction),
    this shows that $L_P(\floor{s_0}) < L_P(s_0)$,
    a contradiction.
\end{proof}

Finally,
using unimodular transforms,
we may show the characterization for rational polytopes.

\begin{restatetheorem}{thm:characterization}
    Let $P$ be a rational polytope
    written as in~\eqref{intersection-representation}.
    Then $L_P(s) = L_P(\floor s)$
    if and only if all $a_i$ are integers
    and all $b_i$ are either $0$ or $1$.
\end{restatetheorem}

\begin{proof}
    The ``if'' part is Theorem~\ref{thm:if-part},
    so assume that $L_P(s) = L_P(\floor s)$.
    By Proposition~\ref{thm:weak-only-if-part},
    we must have $0 \in P$.
    If $P$ is full-dimensional,
    we just need to apply Theorem~\ref{thm:only-if-part}:
    if any of the $a_i$ is non-integer,
    then it must be rational (because $P$ is rational)
    and the corresponding $b_i$ must be zero,
    so we may just multiply the inequality
    by the $\lcm$ of the denominators of the coordinates of $a_i$.
    Thus, assume that $P$ is not full-dimensional.

    Let $L = \aff P$,
    the affine hull of $P$.
    Since $P$ contains the origin,
    $L$ is a vector space;
    since $P$ is rational,
    $L$ is spanned by integer points.
    Let $\dim P$ be the dimension of $P$ (and of $L$)
    and let $d$ be the dimension of the ambient space
    (so that $P \subseteq \mathbb R^d$).
    Then there is a unimodular transform $M$
    which maps $L$ to $\mathbb R^{\dim P} \times \{0\}^{d - \dim P}$.

    If $P$ is contained in the half-space
    \begin{equation*}
        \{x \in \mathbb R^d \mid \langle a, x \rangle \leq b\},
    \end{equation*}
    then $MP$ is contained in the half-space
    \begin{equation*}
        \{x \in \mathbb R^d \mid \langle M^{-t}a, x \rangle \leq b\},
    \end{equation*}
    so applying unimodular transforms
    don't change neither the hypothesis nor the conclusions.
    Thus,
    let $Q$ be the projection of $MP$ to $\mathbb R^{\dim P}$;
    then $Q$ is a full-dimensional polytope which satisfies $L_Q(s) = L_P(s)$,
    so we may apply Theorem~\ref{thm:only-if-part} to conclude the proof.
\end{proof}

We finish this section by remarking that
the hypothesis of $P$ being either full-dimensional or rational
is indeed necessary.
For example,
if $H$ is the set of vectors which are orthogonal to
$(\ln 2, \ln 3, \ln 5, \ln 7, \dots, \ln p_d)$
(where $p_d$ is the $d$th prime number)
and $x = (x_1, \dots, x_d)$ is an integral vector,
$x \in H$ if and only if
\begin{equation*}
    x_1 \log 2 + \dots x_d \log p_d = 0,
\end{equation*}
which (by applying $e^x$ to both sides)
we may rewrite as
\begin{equation*}
    2^{x_1} 3^{x_2} \dots p_d^{x_d} = 1,
\end{equation*}
which is only possible if $x_1 = \dots = x_d = 0$.
That is,
the only integral point in $H$ is the origin.
So,
if $P \subseteq H$,
then $L_P(s)$ will be a constant function,
regardless of any other assumption over $P$.

\section{Interiors of polytopes}
\label{sec:interiors-of-polytopes}

It is interesting to note there are results
similar to Theorems \ref{thm:characterization} and~\ref{thm:only-if-part},
but for interiors of polytopes.
More precisely:

\begin{theorem}
    Let $P$ be a full-dimensional polytope
    written as in~\eqref{intersection-representation}.
    Then $L_{P^\circ}(s) = L_{P^\circ} (\ceil s)$ for all $s \geq 0$
    if and only if all $b_i$ is either $0$ or $1$,
    and $a_i$ is integral whenever $b_i = 1$.
\end{theorem}

In other words,
$L_P(s) = L_P(\floor s)$ if and only if $L_{P^\circ}(s) = L_{P^\circ}(\ceil s)$.

\begin{proof}
    The proof is very similar to the proofs of the theorems~\ref{thm:if-part},
    \ref{thm:weak-only-if-part}, and~\ref{thm:only-if-part},
    so only the needed changes will be stated.

    For Theorem~\ref{thm:if-part},
    the case $b_i = 0$ is left unchanged,
    and when $b_i = 1$ we need to use $[n < x] = [n < \ceil x]$ to conclude that,
    in all cases,
    $[\langle a_i, x \rangle < s b_i]
    = \big[\langle a_i, x \rangle < \ceil s b_i\big]$.

    For Proposition~\ref{thm:weak-only-if-part},
    we can use the same $u$, $s_0$ and $x_0$,
    but now we will show $L_{P^\circ}(s_0) < L_{P^\circ}(s_0 - \varepsilon)$.
    A point $x$ which is in the interior of $s_0 P^\circ$
    satisfy all linear restrictions of the form $\langle x, a_i \rangle < s_0 b_i$.
    For all sufficiently small $\varepsilon > 0$
    we have $\langle x, a_i \rangle < (s_0 - \varepsilon) b_i$,
    and so every integer point in $s_0 P^\circ$
    will also be present on $(s_0 - \varepsilon) P^\circ$.
    Now the linear restriction $\langle u, x \rangle > 1$
    ``shrinks'' to $\langle u, x \rangle > 1 - \varepsilon$,
    and so the point $x_0$
    (which satisfy $\langle u, x_0 \rangle = 1$)
    will be contained in $(s_0 - \varepsilon) P^\circ$,
    and thus $L_{P^\circ}(s_0 - \varepsilon) > L_{P^\circ}(s_0)$.

    And for Theorem~\ref{thm:only-if-part},
    we again can use the same $s_0$ and $x_0$ and also the face $F$,
    but now we will dilate $P^\circ$ instead of shrinking.
    As $sP \subseteq s'P$ for all $s \leq s'$,
    we also have $sP^\circ \subseteq s' P^\circ$ for all $s \leq s'$.
    Now $x_0 \in s_0 F$,
    so for all $\varepsilon > 0$
    we will have $x_0 \in (s_0 + \varepsilon)P^\circ$,
    which thus show $L_P(s_0) < L_P(s_0 + \varepsilon)$.
\end{proof}

Thus,
using essentially the same proof as of Theorem~\ref{thm:characterization},
we have the following characterization of semi-reflexive polytopes.

\begin{theorem}
    Let $P$ be a rational polytope.
    Then $P$ is semi-reflexive
    if and only if
    $L_{P^\circ}(s) = L_{P^\circ} (\ceil s)$ for all real $s \geq 0$.
\end{theorem}

\section{Relation with reflexive polytopes}
\label{sec:relation-with-reflexive}

The hyperplane representation of Theorem~\ref{thm:characterization}
may be rewritten in matricial form as follows:
\begin{equation}
    P = \{x \in \mathbb R^d \mid A_1 x \leq \mathbf 1, A_2 x \leq 0 \},
    \label{eq:matricial-representation}
\end{equation}
where $\mathbf 1 = (1, \dots, 1)$ is the all-ones vector,
and $A_1$ and $A_2$ are integer matrices.
We will use this representation to relate semi-reflexive and reflexive polytopes.

There are several equivalent definitions of reflexive polytopes
(see e.g.~\cite[p.~97]{ccd}).
We mention two of them.

First, $P$ is a reflexive polytope if it is an integer polytope
which may be written as
\begin{equation*}
    P = \{x \in \mathbb R^d \mid A x \leq \mathbf 1 \},
\end{equation*}
where $\mathbf 1 = (1, \dots, 1)$ is the all-ones vector
and $A$ is an integer matrix.
This definition make it obvious that
every reflexive polytope is also semi-reflexive.

The second definition uses the dual $P^*$ of a polytope $P$,
defined by
\begin{equation*}
    P^* = \{x \in \mathbb R^d \mid
            \langle x, y \rangle \leq 1 \text{ for all $y \in P$}
        \}.
\end{equation*}
Then $P$ is reflexive if and only if both $P$ and $P^*$ are integer polytopes.

We have the following relation between reflexive and semi-reflexive polytopes.

\begin{theorem}
    $P$ is a reflexive polytope
    if and only if
    both $P$ and $P^*$ are semi-reflexive polytopes.
\end{theorem}

\begin{proof}
    If $P$ is reflexive,
    then $P^*$ is also reflexive,
    and thus both are semi-reflexive.

    Now,
    suppose that $P$ and $P^*$ are semi-reflexive polytopes.
    This contains the implicit assumption that $P^*$ is bounded,
    and thus $P$ must contain the origin in its interior.
    Therefore,
    by Theorem~\ref{thm:characterization},
    we must have
    \begin{equation*}
        P = \{x \in \mathbb R^d \mid A x \leq \mathbf 1 \}
    \end{equation*}
    for some integer matrix $A$.
    (The fact that $0$ is in the interior of $P$
    allowed us to ignore $A_2$
    in the representation~\eqref{eq:matricial-representation}.)

    So, 
    we just need to show that $P$ has integer vertices.
    Since $(P^*)^* = P$,
    we may apply the same reasoning to $P^*$ to write
    \begin{equation*}
        P^* = \bigcap_{i = 1}^n \{x \in \mathbb R^d \langle a_i, x \rangle \leq 1 \}
    \end{equation*}
    for certain integers $a_1, \dots, a_n$.
    But these $a_i$ are precisely the vertices of $P$,
    being, thus, an integral polytope.
\end{proof}

\bibliographystyle{plain}
\bibliography{bib}

\end{document}